%% file: main.tex
\documentclass[letterpaper, 10 pt, conference]{ieeeconf}  

\IEEEoverridecommandlockouts                              

\usepackage{times} 
\usepackage{amssymb}  
\usepackage[normalem]{ulem}
\usepackage{amsmath}
\usepackage{algorithm}
\usepackage{array}
\usepackage{textcomp}
\usepackage{url}
\usepackage{verbatim,dsfont}
\usepackage{graphicx,color}
\usepackage{cite}
\usepackage{algpseudocode}
\usepackage{soul}
\title{\LARGE \bf
Optimally Linearizing Power Flow Equations for Improved \\Power System Dispatch
}

\author{Yuhao Chen and Manish K. Singh
\thanks{*This work was partially supported by the U.S. NSF under grant 2453151.}
\thanks{Yuhao Chen and Manish K. Singh are with the Department of Electrical and Computer Engineering, University of Wisconsin--Madison, Madison, Wisconsin, USA (email: {\tt\footnotesize \{yuhao.chen,manish.singh\}@wisc.edu}).
        }%
}

\input{StylesMKS}

\begin{document}

\maketitle
\thispagestyle{empty}
\pagestyle{empty}

\begin{abstract}
Managing power grids with the increasing presence of variable renewable energy-based (distributed) generation involves solving high-dimensional optimization tasks at short intervals. Linearizing the AC power flow (PF) constraints is a standard practice to ease the computational burden at the cost of hopefully acceptable inaccuracies. However, the design of these PF linearizations has traditionally been agnostic of the use case. Towards bridging the linearization-application gap, we first model the complete operational sequence needed to implement optimal power flow (OPF) decisions on power systems and characterize the effect of PF linearization on the resulting steady-state system operation. We then propose a novel formulation for obtaining optimal PF constraint linearizations to harness desirable system-operation attributes such as low generation cost and engineering-limit violations. To pursue the optimal PF linearization, we develop a gradient-based approach backed by sensitivity analysis of optimization routines and AC PF equations. Numerical tests on the IEEE 39-bus system demonstrate the capabilities of our approach in traversing the cost-optimality vs operational feasibility trade-off inherent to OPF approximations.

\end{abstract}

\section{INTRODUCTION}
Power system operation involves a gamut of decision-making tasks with timelines varying from minutes to decades. Gloriously placed at the center of these problems are the AC power flow (PF) equations that relate the complex-valued power injections to bus voltages given the power network topology and impedances. The decision-making tasks are thus often referred to as AC optimal power flow (OPF). The PF equations form a nonlinear algebraic system that renders AC OPF tasks intractable for large systems. The proliferation of highly variable resources increases the problem dimensions and poses stricter time budgets for solvers, thus further exacerbating the computational challenges. The past two decades have featured a splendid body of research on alleviating these computational challenges through convex relaxations, restrictions, and linear approximations; see~\cite{Molzahn-Hiskens-PFSurvey} for an overview.

One prominent approach to obtaining numerical tractability is to approximate AC OPF by linearizing the PF constraints. Most linearization approaches belong to two main categories. Approaches in the first category rely on engineering insights (such as low line resistances, small angle deviations, and nominal voltage magnitudes) to drop non-linear terms in PF equations~\cite{Stott2009DCPFrevisited,BaranWu89CapPlacement,BernsteinCDCgeneralizedLDF}. The second category uses the first-order Taylor approximation for a fixed operating point~\cite{Siaraj2015PFLinearization,Bolognani2015implicit}. We refer readers to \cite[Ch. 5]{Molzahn-Hiskens-PFSurvey} for a rich survey on PF linearizations. Linearization quality is typically assessed based on the error between the actual power injections derived by the exact AC PF model and the power injections computed by the linearized one. One can statistically analyze these errors based on historical, predicted, or randomly generated operating points. Thus, recent works have parted from the traditional emphasis on a single linearization point to design model-based and data-based PF linearizations that accurately approximate AC PF over a distribution of operating conditions~\cite{Taheri-Molzahn-TPEC24,taheri2024optimized,Jiaqi22TSG_PWL}.

Interestingly, most existing evaluation and design approaches for PF linearization do not take the end-use (e.g., simplifying OPF) into consideration. Acknowledging that the merit of a PF linearization depends on how well it serves the downstream applications, reference~\cite{Baker-Kar-22GM} numerically compared seven PF linearization techniques when used to simplify OPF. It reports that the choice of linearization significantly impacts the optimality and feasibility of approximated minimizers. Among the rare efforts towards designing PF linearization for use in a specific OPF setting, Reference~\cite{LygerosTPWRS19_OptLin} develops semidefinite programs (SDP) to optimally find a linearization point that minimizes the expected AC OPF constraint violations. However, the developed approach offers limited flexibility and is computationally restricted by the abilities of SDP solvers. Recently, a more generalizable framework was put forth in~\cite{taheri2024ac}, where the linearization coefficients are optimally determined to reduce the Euclidean norm of the difference between the minimizers of AC OPF and that of the approximated DC OPF. Such a framework involves a supervised-learning-type approach that requires true AC OPF minimizers for a training set and uses bilevel optimization to tune linearization coefficients. However, we observe that the specific desired qualities from the minimizer of an approximated OPF are cost optimality and AC PF feasibility. These aspects are not necessarily captured by the Euclidean norm of the difference from AC OPF's minimizer. In this work, we show that one can directly optimize for the desired attributes, thereby eliminating the need to solve AC OPF to build the training set.

In this work, we develop an approach for optimal PF linearization while explicitly modeling the impact of linearization coefficients on the steady-state system operation. Therefore, our model subsumes the effect of postprocessing schemes and grid controls that are necessary for implementing the decisions obtained from DC OPF; see Fig.~\ref{fig:overview}. The novel contributions of this work include: \emph{i)} Developing a model that characterizes the impact of PF linearization coefficients on the steady-state operation of power systems; \emph{ii)} Formulating a novel application-informed optimal PF linearization problem; \emph{iii)} Conducting sensitivity analysis for individual subsystems in the grid operation architecture of Fig.~\ref{fig:overview}; and \emph{iv)} Presenting an analysis of the trade-off between operational feasibility and cost optimality that can be navigated by adjusting a weight parameter in the proposed formulation. The manuscript is organized as follows: Section~\ref{sec:overview} introduces the problem setup and the proposed solution approach with minimal notational overhang; the detailed modeling of DC OPF and steady-state operation of the AC power system is provided in Section~\ref{sec:model}; the sensitivity analysis that serves as the bedrock of our gradient-based approach is presented in Section~\ref{sec:SA}; numerical tests and empirical analysis is presented in Section~\ref{sec:tests} before providing concluding remarks.

\emph{Notation:} Lower- (upper-) case boldface letters denote column vectors (matrices). For a vector $\bx$, its $n$-th entry is denoted as $x_n$. The $(i,j)$-th entry of a matrix $\bA$ is represented by $A_{ij}$. The symbol $^\top$ stands for transposition, and inequalities are understood element-wise. A vector of all ones is denoted by $\bone$; a vector/matrix of all zeros is represented as $\bzero$. The identity matrix is denoted as $\bI$, and $\be_n$ is the $n$-th canonical vector. The dimensions for $(\bone, \bzero, \bI, \be_n)$ should be clear from the context. The operator $|\cdot|$ yields: the absolute value for real-valued arguments; and the cardinality when the argument is a set. Complex quantities are denoted using $\jmath:=\sqrt{-1}$. Operator $\diag()$ yields a diagonal matrix by placing its vector argument as the main diagonal. The indicator function is denoted by $\mathds{1}(\cdot)$ and acts element-wise. Gradients are represented using the numerator-layout; thus, for vectors $\bx$ and $\by$, the $(i,j)$-th entry of $\nabla_{\bx}\by$ is $\partial y_i/\partial x_j$.  

\section{OVERVIEW OF THE PROBLEM SETUP AND THE PROPOSED SOLUTION}\label{sec:overview}
\subsection{Problem Setup}
Consider the task of determining least-cost generation dispatch in a bulk power system with known complex power demand $\bp^\mrd+\jmath\bq^\mrd$. A power generation vector $\bp^\mrg+\jmath\bq^\mrg$ is admissible only if it allows the governing steady-state AC PF equations to have a solution. Additionally, the resulting operating point (network voltages and line power flows) must satisfy engineering and regulatory limits of the physical infrastructure. Abstractly representing an operating point as $\bpi$, and the feasible operating set  as $\bPi$, the generation dispatch AC OPF task involves solving
\begin{align}
\min_{\bp^\mrg,\bq^\mrg,\bpi\in\bPi}~&c(\bp^\mrg)~\notag\\
\textrm{such that,}~&~(\bp^\mrd,\bq^\mrd,\bp^\mrg,\bq^\mrg,\bpi)~\text{satisfy AC PF}.\notag
\end{align}
The generation cost is typically a convex quadratic function of only the active power generation. However, the nonconvexity of the set of AC PF-feasible generation renders the above task nonconvex. To avoid the computational burden for large power systems, linear power-flow approximations and additional optimization model simplifications are often called upon. One such frequently encountered simplification is referred to as the DC Optimal Power Flow, which results in a mapping ${\verb|DCOPF|}:\bp^\mrd\rightarrow\bp^\mrg$. The mapping implicitly depends on the coefficients of the linearized PF equations that replace AC PF constraints while formulating DC OPF; see Fig.~\ref{fig:overview}. Let $\bPsi$ be a vectorized collection of the linearization coefficients. The DC OPF output thus depends on $\bPsi$, i.e., $\bp^\mrg_{\mrDC}={\verb|DCOPF|}(\bp^\mrd;\bPsi)$.

   \begin{figure}[t]
      \centering
      \includegraphics[scale=1]{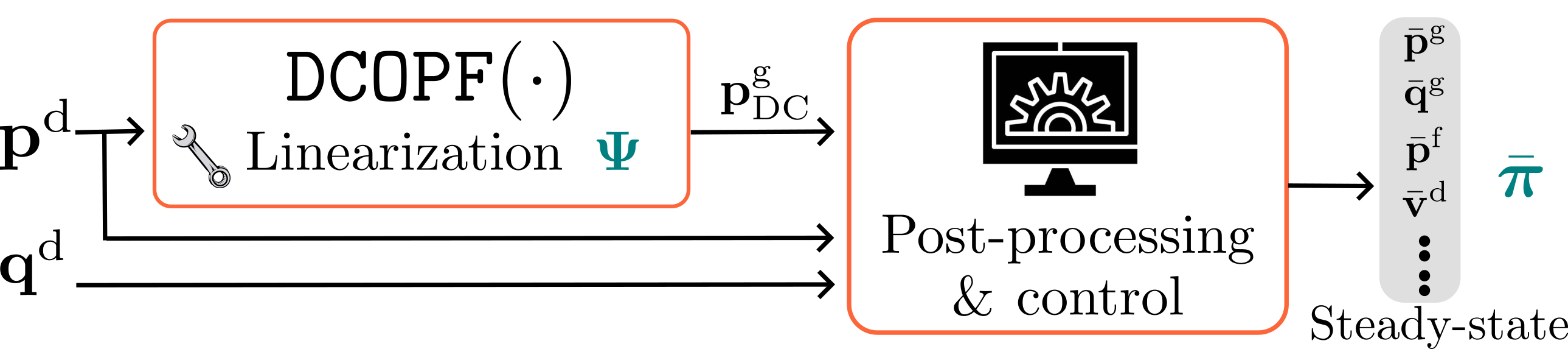}
      \caption{Schematic of the considered power system operational pipeline. This work characterizes the dependence of steady-state operating point $\bar{\bpi}$ on the PF linearization coefficients $\bPsi$ that are used to simplify OPF formulation. The main contribution of this work is to develop an approach that optimizes $\bPsi$ to improve the cost-optimality and operational feasibility of $\bar{\bpi}$.}\vspace{-1.5em}
      \label{fig:overview}
   \end{figure}

While the convexity of DC OPF makes it popular, it comes with obvious limitations. For instance, it is incognizant of the reactive power demand $\bq^\mrd$ and does not output $\bq^\mrg$. To make the situation worse, there may not exist a $\bq^\mrg$ such that $(\bp^\mrd,\bq^\mrd,\bp^\mrg_\mrDC,\bq^\mrg)$ admits an AC PF solution. Furthermore, even if the solution exists, the resulting operating point $\bpi$ may not be in the feasible set $\bPi$, for instance, there may be a line flow limit violation. To overcome the first challenge of AC PF infeasibility of $\bp^\mrg_\mrDC$, system operators often have post-processing heuristics and grid control schemes that map $\bp^\mrg_\mrDC\mapsto (\bar{\bp}^\mrg,\bar{\bq}^\mrg)$, such that $(\bp^\mrd,\bq^\mrd,\bar{\bp}^\mrg,\bar{\bq}^\mrg)$ admits an AC PF solution. The resulting operating point $\bar{\bpi}$ may still violate the engineering limits, implying $\bar{\bpi}\notin\bPi$. Additional operational practices and control schemes may be deployed to restore operational feasibility in such cases. The aforementioned post-processing steps inevitably compromise the cost optimality of the generation dispatch. In this work, we ascribe cost optimality and feasibility of the steady-state operation as two desired qualities for a DC OPF model. Accounting for the dependence of DC OPF on PF linearization coefficients $\bPsi$, we seek to solve 
\begin{equation}\label{eq:bilevelE}
  \min_{\bPsi} \mathbb{E}_{(\bp^\mrd,\bq^\mrd)} c(\bar{\bp}^\mrg)+w\dist(\bar{\bpi},\bPi),
\end{equation}
where $w$ is a scalar parameter balancing the two objectives and $\dist()$ is a distance metric quantifying the extent of violations of the engineering limits. Figure~\ref{fig:overview} provides an overview of the problem setup. In formulating~\eqref{eq:bilevelE}, we assume the mapping $\bp^\mrg_\mrDC\mapsto (\bar{\bp}^\mrg,\bar{\bq}^\mrg)$ to be fixed and known. Section~\ref{sec:ACPF} describes one such practical mapping. We next briefly discuss the solution approach to tackle~\eqref{eq:bilevelE}.

\subsection{Proposed Solution Approach}
Power-system operators typically have empirical knowledge of demand distribution through scenarios. Therefore, using the sample mean approximation in~\eqref{eq:bilevelE}, we recast the task of optimal linearization as 
\begin{equation}\label{eq:bilevel}
  \min_{\bPsi} \sum_{s=1}^S c(\bar{\bp}_s^\mrg)+w\dist(\bar{\bpi}_s,\bPi),
\end{equation}
where $S$ demand scenarios $\{\bp_s^\mrd,\bq_s^\mrd\}_{s=1}^S$ are used to evaluate the optimal linearization coefficients. For each $(\bp_s^\mrd,\bq_s^\mrd)$, the corresponding $\bar{\bp}^\mrg_s$ and $\bar{\bpi}_s$ are obtained as shown in Fig.~\ref{fig:overview}. Computing $\bp^\mrg_{\mrDC}$ requires solving the DC OPF problem, making~\eqref{eq:bilevel} a bilevel optimization task. The lower-level problem obtains the linearized optimal power flow solutions $\bp^\mrg_{\mrDC}$ for each demand scenario. The upper-level problem includes the post-processing procedure based on the lower-level solution, and minimizes the cost \eqref{eq:bilevel}. The detailed formulation of the two levels will be discussed in section~\ref{sec:model}. 

For related bilevel problem settings, gradient-based methods have recently yielded promising performance~\cite{taheri2024ac}. Inspired by these findings, we develop a gradient-descent solver for~\eqref{eq:bilevel}. To that end, computing the gradients $\nabla_{\bPsi}c(\cdot)$ and $\nabla_{\bPsi}\dist(\cdot)$ is required. Back-propagating through the workflow of Fig.~\ref{fig:overview}, we subsequently compute \emph{i)} Gradient of the cost function~\eqref{eq:bilevelE} with respect to $\bar{\bp}^g$ and $\bar{\bpi}$; \emph{ii)} Jacobians $\nabla_{\bp^\mrg_\mrDC}\bar{\bp}^\mrg$ and $\nabla_{\bp^\mrg_\mrDC}\bar{\bpi}$; and finally, \emph{iii)} the DC OPF sensitivities $\nabla_{\bPsi}\bp^\mrg_\mrDC$. The second and third steps are accomplished as follows:
\begin{itemize}
    \item In Section~\ref{sec:ACPF}, we characterize the mapping $\bp^\mrg_\mrDC\mapsto (\bar{\bp}^\mrg,\bar{\bpi})$ using a distributed-slack-based AC power flow solver that succinctly captures power-system operation under the widely deployed hierarchical grid control structure~\cite{Dhople20Slack}. This allows us, in Section~\ref{sec:SCACPF}, to compute $\nabla_{\bp^\mrg_\mrDC}\bar{\bp}^\mrg$ and $\nabla_{\bp^\mrg_\mrDC}\bar{\bpi}$ using implicit differentiation of the adopted AC power flow equations.
    \item In Section~\ref{sec:DCOPF}, we instantiate the DC OPF formulation parameterized by the linearization coefficients $\bPsi$. Computing $\nabla_{\bPsi}\bp^\mrg_\mrDC$ requires conducting a sensitivity analysis. Specifically, under certain conditions identified in Section~\ref{sec:SADCOPF}, one can use implicit differentiation of the Karush–Kuhn–Tucker (KKT) conditions for the DC OPF problem to compute $\nabla_{\bPsi}\bp^\mrg_\mrDC$.
\end{itemize}
Admittedly, computing the above sensitivities for each scenario per gradient step of solving~\eqref{eq:bilevel} is numerically daunting. To scale these computations (to some extent), we note that while we require the product $(\nabla_{\bp^\mrg_{\mrDC}}c(\cdot))^\top\nabla_{\bPsi}\bp^\mrg_{\mrDC}$, for instance, we do not necessarily require an explicit computation of $\nabla_{\bPsi}\bp^\mrg_{\mrDC}$. Capitalizing on this structure, Section~\ref{sec:SADCOPF} presents a technique to use directional derivatives that reduces computational costs significantly.

\section{MODELING}\label{sec:model}
Consider a single-phase equivalent model for a bulk power system represented as an undirected graph $(\mcN,\mcE)$. The nodes indexed as $n\in\mcN=\{1,\dots,N\}$ correspond to buses, and the edges $e\in\mcE$ correspond to transmission lines.  Assigning arbitrary directionality to edges, an edge $e\in\mcE$ can be denoted as $e=(m,n)$ if it runs from node $m$ to $n$. Denote the impedance and sending end active power flow for line $e$ by $r_e+\jmath x_e$ and $p^\mrf_e$, respectively. Let the complex voltage at bus $n\in\mcN$ be $v_n\angle\theta_n$. The network topology is captured by the $E\times N$ branch-bus incidence matrix $\bA$ with entries
	\begin{equation}\label{eq:A}
	A_{e,k}:=
	\begin{cases}
	+1&,~k=m\\
	-1&,~k=n\\
	0&,~\text{otherwise}
	\end{cases}~\forall~e=(m,n)\in\mcN.
	\end{equation}
Without loss of generality, we assume that all buses $n\in\mcN$ have (potentially zero) demands $p^\mrd_n+\jmath q_n^\mrd$, while the first $N_\mrg$ buses host generators; denote the set of generator buses as $\mcN_\mrg=\{1,\dots,N_\mrg\}\subset\mcN$. Partition the voltage-magnitude vector as $\bv=[(\bv^\mrg)^\top~(\bv^\mrd)^\top]^\top$, where $\bv^\mrg\in\mathbb{R}^{N_\mrg}$. Define matrix $\bF_\mrg=[\bI_{N_\mrg}~~\bzero]^\top$ that maps the generators to their respective buses, such that the nodal power injections are 
\begin{subequations}\label{eq:nodes}
    \begin{align}
        \bp&=\bF_\mrg \bp^\mrg-\bp^\mrd,\label{seq:nodeA}\\
        \bq&=\bF_\mrg \bq^\mrg-\bq^\mrd,
    \end{align}
\end{subequations}
where vector $\bp^\mrg+\jmath\bq^\mrg$ denotes power generation. Let the generation cost at bus $n$ be $c_n(p^\mrg_n)^2$, implying the total cost of generation being $(\bp^\mrg)^\top\bC\bp^\mrg$, where, $\bC=\diag(\{c_n\}_{n=1}^{N_\mrg})$. 
\subsection{DC-OPF Model}\label{sec:DCOPF}
The classical DC power flow approximation dictates
\begin{subequations}\label{eq:DCpf}
    \begin{align}
        \bp^\mrf&=(\diag(\bx))^{-1}\bA\btheta,\\
        \bp&=\bA^\top\bp^\mrf,
    \end{align}
\end{subequations}
where, $\bp^\mrf$ and $\bx$ are the vectors of line active power flows and reactances, respectively. To enhance the flexibility of the above model while retaining linearity, we augment the equations as~\cite{Taheri-Molzahn-TPEC24}
\begin{subequations}\label{eq:DCpf+}
    \begin{align}
        \bp^\mrf&=\bM\btheta+\bgamma,\\
        \bp&=\bA^\top\bp^\mrf+\bb,
    \end{align}
\end{subequations}
where, parameters $\bM\in\mathbb{R}^{E\times N},~\bgamma\in\mathbb{R}^E$, and $\bb\in\mathbb{R}^N$ are linearization coefficients that we will optimally determine, i.e., $\bPsi:=(\bM, \bgamma, \bb)$ in~\eqref{eq:bilevel}. With the modeling above, the DC OPF problem can be formulated as
\begin{subequations}\label{eq:P1}
\begin{align}
\min_{\bp^\mrg}~&~(\bp^\mrg)^\top\bC\bp^\mrg\qquad\qquad\textrm{(P1)}\notag\\
\textrm{s.to}~&~\eqref{seq:nodeA},~\eqref{eq:DCpf+},\\
~&~\bzero\leq\bp^\mrg\leq\bp_{\max}^\mrg,\label{seq:pg}\\
~&~|\bp^\mrf|\leq\bp_{\max}^\mrf,\label{seq:pl}
\end{align}
\end{subequations}
where, constraints \eqref{seq:pg} and \eqref{seq:pl} enforce the power generation and line flow limits. Note that, problem (P1) is parameterized by $\bPsi=(\bM, \bgamma, \bb)$ and describes a mapping from demand $\bp^\mrd$ to optimal generation $\bp^\mrg_{\mrDC}$, thus defining the mapping ${\verb|DCOPF|(\cdot)}$ in Fig.\ref{fig:overview}. Problem (P1) will therefore serve as the lower-level problem in solving~\eqref{eq:bilevel}. 

\subsection{AC Power Flow Model with Distributed Slack}\label{sec:ACPF}
We assume the following scheme of power system operation: \emph{Step-1)} System operator solves (P1) for a given demand $\bp^\mrd$; \emph{Step-2)} The obtained optimal dispatch $\bp^\mrg_\mrDC$ is used as active power setpoints for the generators. The voltage-magnitude references are set to $\bv^\mrg_\circ$; \emph{Step-3)} Based on the generator setpoints $(\bp^\mrg_\mrDC,\bv^\mrg_\circ)$, and demand $(\bp^\mrd,\bq^\mrd)$, the AC steady-state power system operating point is determined by the primary- and secondary-control schemes. The steady state quantities are denoted as $(\bar{\bp}^\mrg, \bar{\bq}^\mrg, \bar{\bv}^\mrd, \bar{\btheta}, \bar{\bp}^\mrf)$. Following a deterministic setting, we assume the demand $(\bp^\mrd,\bq^\mrd)$ is constant through the above steps. Conveniently, the system operating point determined by the above operating procedure can be accurately characterized by a system of nonlinear equations referred to as the distributed-slack bus formulation for AC power flow~\cite{Dhople20Slack}. 

The AC power flow equations are often expressed using the admittance matrix $\bY:=\bA^\top[\diag(\br+\jmath\bx)]^{-1}\bA$, where the real and imaginary parts can be separated as $\bY=\bG+\jmath\bB$. Given demand $(\bp^\mrd,\bq^\mrd)$ and generator setpoint $(\bp^\mrg_\mrDC,\bv^\mrg_\circ)$, the steady-state operating point satisfies 
\begin{subequations}\label{eq:ACPF}
\begin{align}
\bar{\bp}^\mrg&=\bp^\mrg_\mrDC + \balpha~\zeta\label{seq:ACPFa}\\
\bar{\bv}^\mrg&=\bv^\mrg_\circ\label{seq:ACPFb}\\
\bar{\bp}&=\bF_\mrg \bar{\bp}^\mrg-\bp^\mrd\label{seq:ACPFc}\\
\bar{\bq}&=\bF_\mrg \bar{\bq}^\mrg-\bq^\mrd\label{seq:ACPFd}\\
\bar{p}_n&=\bar{v}_n \sum_{k \in \mathcal{N}} \bar{v}_k ( G_{nk}\cos\bar{\theta}_{nk} + B_{nk}\sin\bar{\theta}_{nk}),~\forall n\label{seq:ACPFe}\\
\bar{q}_n&=\bar{v}_n \sum_{k \in \mathcal{N}} \bar{v}_k ( G_{nk}\sin\bar{\theta}_{nk} - B_{nk}\cos\bar{\theta}_{nk}),~\forall n\label{seq:ACPFf}\\
\bar{\theta}_1&=0,~\bar{\theta}_{nk}=\bar{\theta}_n-\bar{\theta}_k \forall (n,k)\in\mcE,\label{seq:ACPFg}
\end{align}
\end{subequations}
where, vector $\balpha\in[0,1]^{N_\mrg}$ contains the participation factors,
$$\balpha=\frac{\bp^\mrg_{\max}}{\bone^\top\bp^\mrg_{\max}},$$
and $\zeta$ is an unknown scalar that accounts for the active power losses. After one solves~\eqref{eq:ACPF}, the active power flows on line $e=(m,n)\in\mcE$ can be computed as 
\begin{equation}\label{eq:pl}
    \bar{p}^\mrf_e=G_{mn}\bar{v}_m^2-\bar{v}_m\bar{v}_n(G_{mn}\cos\bar{\theta}_{mn}+B_{mn}\sin\bar{\theta}_{mn}).
\end{equation}
\subsection{Optimal Linearization Objective}\label{sec:upper_obj}

We quantify the desirability of linearization coefficients $\bPsi$ based on the cost and feasibility of the ultimate steady state quantities $(\bar{\bp}^\mrg, \bar{\bq}^\mrg, \bar{\bv}^\mrd, \bar{\bp}^\mrf)$. Specifically, we measure optimality using the cost of steady-state generation $(\bar{\bp}^\mrg)^\top\bC\bar{\bp}^\mrg$, and use violations in active power generation and line flows (cf.~\eqref{seq:pg}-\eqref{seq:pl}) as the measure of infeasibility. The objective function in~\eqref{eq:bilevel} can thus be instantiated (for one scenario) as
\begin{equation}\label{eq:loss}
    \ell=(\bar{\bp}^\mrg)^\top\bC\bar{\bp}^\mrg+w\bone^\top\begin{bmatrix}
    \max(\bzero,\bar{\bp}^\mrg-\bp^\mrg_{\max})\\\max(\bzero,|\bar{\bp}^\mrf|-\bp^\mrf_{\max})
\end{bmatrix}.
\end{equation}
In practice, one could also penalize violations in load voltage magnitudes and generator reactive power injections. The penalty weights could be tuned differently for each constraint based on criticality. However, we proceed with the structural choice of~\eqref{eq:loss} for expositional ease without loss of generality. To recapitulate, determining the optimal PF linearization $\bPsi=(\bM,\bgamma,\bb)$ involves minimizing the cost function~\eqref{eq:loss} constrained by \eqref{eq:ACPF}-\eqref{eq:pl}, and the instances of the lower-level problem (P1).

\section{SENSITIVITY ANALYSIS}\label{sec:SA}
This section provides the sensitivity analysis for the blocks in Fig.~\ref{fig:overview} in reverse order. We first compute $\nabla_{\bar{\bp}^\mrg}\ell$ and $\nabla_{\bar{\bp}^\mrf}\ell$ from~\eqref{eq:loss}. Next, we will delineate the steps involved in computing $\nabla_{\bp^\mrg_\mrDC}\bar{\bp}^\mrg$ and $\nabla_{\bp^\mrg_\mrDC}\bar{\bp}^\mrf$ using the AC PF model of Section~\ref{sec:ACPF}. Finally, sensitivity analysis for the DC OPF (P1) will provide $\nabla_{M_{ij}}\bp^\mrg_\mrDC, \nabla_{\bgamma}\bp^\mrg_\mrDC$, and $\nabla_{\bb}\bp^\mrg_\mrDC$. At the outset, the loss function $\ell$ in~\eqref{eq:loss} is non-differentiable at the generation and flow limits because of the $\max$ operator. With some abuse of notation, we express the sub-gradients as
\begin{subequations}\label{eq:gradL}
\begin{align}
	(\nabla_{\bar{\bp}^\mrg}\ell)^\top&=2\bC\bar{\bp}^\mrg + w\mathds{1}(\bar{\bp}^\mrg\geq\bp^\mrg_{\max})\label{seq:gradLpg}\\
    (\nabla_{\bar{\bp}^\mrf}\ell)^\top&=w[\mathds{1}(\bar{\bp}^\mrf\geq\bp^\mrf_{\max})-\mathds{1}(\bar{\bp}^\mrf\leq-\bp^\mrf_{\max})],\label{seq:gradLpf}
\end{align}
\end{subequations}
where $\mathds{1}(\cdot)$ is the indicator function that applies entry-wise.

\subsection{Sensitivity Analysis for AC Power Flow with Distributed Slack Bus Formulation}\label{sec:SCACPF}
Building on~\eqref{eq:gradL} and using total derivatives, we aim at computing
\begin{equation}\label{eq:gradLpgdc}
\nabla_{\bp^\mrg_\mrDC}\ell=\nabla_{\bar{\bp}^\mrg}\ell\nabla_{\bp^\mrg_\mrDC}\bar{\bp}^\mrg+\nabla_{\bar{\bp}^\mrf}\ell\nabla_{\bp^\mrg_\mrDC}\bar{\bp}^\mrf.
\end{equation}
Thus, we next derive the sensitivity of $(\bar{\bp}^\mrg,\bar{\bp}^\mrf)$ with respect to $\bp^\mrg_\mrDC$ using~\eqref{eq:ACPF}-\eqref{eq:pl}. Note that, $(\bar{\bp}^\mrg,\bar{\bp}^\mrf)$ are explicit functions of $(\zeta,\bar{\bv},\bar{\btheta})$; cf.~\eqref{seq:ACPFa}, \eqref{eq:pl}. Hence, we focus on computing sensitivities of $(\zeta,\bar{\bv},\btheta)$ with respect to $\bp^\mrg_\mrDC$. Additionally, since $\theta_1=0$ and $\bar{\bv}^\mrg=\bv_\circ^\mrg$ are constants, the sought sensitivities are limited to $\nabla_{\bp^{\mrg}_{\mrDC}}\zeta$, $\nabla_{\bp^{\mrg}_{\mrDC}} \bar{\theta}_k$ for $k=2,\dots,N$, and $\nabla_{\bp^{\mrg}_{\mrDC}} \bar{\bv}^\mrd$. Let $\Check{\btheta}:=\{\bar{\theta}_n\}_{n=2}^N$.

Substituting~\eqref{seq:ACPFa}-\eqref{seq:ACPFd} and \eqref{seq:ACPFg} in \eqref{seq:ACPFe}-\eqref{seq:ACPFf}, we get $2N$ equations in $(\zeta,\bar{\bv}^\mrd,\bar{\bq}^\mrg, \Check{\btheta})$. Since in this work, we do not require computing $\bar{\bq}^\mrg$, we can drop the corresponding equations from~\eqref{seq:ACPFf} to finally obtain $2N-N_\mrg$ equations in $2N-N_\mrg$ unknowns. Let us denote these equations as $\{f_n\}_{n=1}^{2N-N_{\mrg}}$. One can use any off-the-shelf non-linear solver for this system of equations. In our numerical tests, we use the MATLAB function ${\verb|fsolve|}$ to find the solution, and compute the Jacobian,
$$\bJ=\begin{bmatrix}
    \frac{\partial f_1}{\partial\zeta}&\frac{\partial f_1}{\partial\bar{v}_{N_\mrg+1}}&...&\frac{\partial f_1}{\partial \bar{v}_N} &\frac{\partial f_1}{\partial\theta_2}&...&\frac{\partial f_1}{\partial\theta_N}\\
    :&:&...&:&:&...&:\\
    \frac{\partial f_{2N-N_{\mrg}}}{\partial\zeta}&...&...&... &...&...&\frac{\partial f_{2N-N_{\mrg}}}{\partial\theta_N}
\end{bmatrix}$$
For an infinitesimal change $\bp^\mrg_\mrDC\leftarrow\bp^\mrg_\mrDC+\bdelta_{\bp^{\mrg}_\mrDC}$, the corresponding change in solutions of $\{f_n\}_{n=1}^{2N-N_{\mrg}}$ can be found by solving
\begin{equation}\label{eq:SAACPF}
\bJ\begin{bmatrix}
    \delta_{\zeta}\\\bdelta_{\bar{\bv}}\\\bdelta_{\bar{\btheta}}
\end{bmatrix}=\begin{bmatrix}
    \bF_\mrg\\\bzero
\end{bmatrix}\bdelta_{\bp^{\mrg}_\mrDC}\implies \begin{bmatrix}
    \nabla_{\bp^{\mrg}_{\mrDC}}\zeta\\
    \nabla_{\bp^{\mrg}_{\mrDC}}\bar{\bv}^\mrd\\
    \nabla_{\bp^{\mrg}_{\mrDC}}\Check{\btheta}
\end{bmatrix}=\bJ^{-1}\begin{bmatrix}
    \bF_\mrg\\\bzero
\end{bmatrix}.\end{equation}
Having obtained the gradients in~\eqref{eq:SAACPF}, one can readily evaluate~\eqref{eq:gradLpgdc} by observing
\begin{subequations}\label{eq:total}
    \begin{align}
\nabla_{\bp^\mrg_\mrDC}\bar{\bp}^\mrg&=\bI+\balpha\nabla_{\bp^{\mrg}_{\mrDC}}\zeta\label{seq:totala}\\
\nabla_{\bp^\mrg_\mrDC}\bar{\bp}^\mrf&=\nabla_{\bar{\bv}^\mrd}\bar{\bp}^\mrf\nabla_{\bp^{\mrg}_{\mrDC}}\bar{\bv}^\mrd+\nabla_{\Check{\btheta}}\bar{\bp}^\mrf\nabla_{\bp^{\mrg}_{\mrDC}}\Check{\btheta},\label{seq:totalb}
\end{align}
\end{subequations}
where \eqref{seq:totala} stems from~\eqref{seq:ACPFa}, and $(\nabla_{\bar{\bv}^\mrd}\bar{\bp}^\mrf, \nabla_{\Check{\btheta}}\bar{\bp}^\mrf)$ can be obtained from~\eqref{eq:pl}.
\subsection{Sensitivity Analysis for DC OPF}\label{sec:SADCOPF}
Sensitivity analysis of the DC OPF in (P1) entails computing how the minimizer $\bp^\mrg_\mrDC$ changes with infinitesimal change in linearization parameters $(\bM, \bgamma, \bb)$, i.e., $\nabla_{M_{ij}} \bp^\mrg_\mrDC$, $\nabla_{\bgamma} \bp^\mrg_\mrDC$, and $\nabla_{\bb} \bp^\mrg_\mrDC$. To unclutter the exposition, consider the following abstraction of the quadratic program
\begin{subequations}
\begin{align}
\bchi^\star=\arg\min_{\bchi}~&\bchi^\top\bP\bchi~\qquad(\textrm{P2})\notag\\
\textrm{s.to}~&~\bW\bchi+\bu=\bR\bd\label{seq:QPa}~:~\blambda\\
~&~\bS\bchi+\bT\bd\leq\bu~:~\bmu,\label{seq:QPb}
\end{align}
\end{subequations}
where $\blambda$ and $\bmu$ are the Lagrange multipliers corresponding to \eqref{seq:QPa}-\eqref{seq:QPb}, and matrix $\bP$ is symmetric positive definite. Problem (P1) is an instance of (P2), i.e., one can define the quantities $(\bP,\bW,\bS,\bT,\bchi,\bd,\bu)$ such that (P2) coincides with (P1). The sensitivity analysis goal thus translates to computing $\nabla_{W_{ij}} \bchi^\star,~\forall i, j,$ and $\nabla_{\bu} \bchi^\star$. These sensitivities are known to exist and can be readily computed under certain standard technical conditions stated next~\cite{fiacco1976sensitivity,Conejo06,L2O2021}.
\begin{assumption}\label{as:1}
   Given a tuple of optimal primal/dual variables $(\bchi^\star,\blambda^\star,\bmu^\star)$, the $n^{th}$ constraint in~\eqref{seq:QPb} is active, i.e., $\be_n^\top(\bS\bchi^\star+\bT\bd-\bu)=0$, if and only if $\mu_n>0$. Denote the set of active constraints as $\mcA$. The rows of $\bW$ and vectors $\{\be_n^\top\bS\}_{n\in\mcA}$ are linearly independent.
\end{assumption}
\begin{assumption}\label{as:2}
    Denote the Lagrangian function of (P2) by $L(\bchi,\blambda,\bmu)$. For a subspace orthogonal to the subspace spanned by the gradients of active constraints 
\[\mcZ:=\left\{\bz: \bW\bz=\bzero,~\be_n^\top\bS\bz=0~\forall n\in\mcA\right\}\]
it holds that $\bz^\top\nabla_{\bchi\bchi}^2L\bz>0$ for all $z\in\mcZ\setminus\{\bzero\}$.
\end{assumption}
Under the aforementioned conditions, the following result forms the basis of our sensitivity analysis.
\begin{proposition}[\cite{fiacco1976sensitivity}\cite{L2O2021}]\label{prop:1} Let $(\bchi^\star,\blambda^\star,\bmu^\star)$ denote the optimal primal and dual variables of (P2). Consider an infinitesimal perturbation in problem parameters $\bW\leftarrow\bW+\bdelta_{\bW}$ and $\bu\leftarrow\bu+\bdelta_{\bu}$ and a corresponding change in optimal solution $(\bchi^\star+\bdelta_{\bchi},\blambda^\star+\bdelta_{\blambda},\bmu^\star+\delta_{\bmu})$. Under Assumptions~\ref{as:1} and \ref{as:2}, the perturbations satisfy 
\begin{equation}\label{eq:prop1}
    \bGamma\begin{bmatrix}
        \bdelta_{\bchi}\\
        \bdelta_{\blambda}\\
        \bdelta_{\bmu}
    \end{bmatrix}=-\begin{bmatrix}
        \bdelta_{\bW}^\top\blambda^\star\\
        \bdelta_{\bW}\bchi^\star+\bdelta_{\bu}\\
        -\diag(\bmu^\star)\bdelta_{\bu}
    \end{bmatrix}
\end{equation}
    where,
    $$\bGamma:=\begin{bmatrix}
        2\bP & \bW^\top & \bS^\top\\
        \bW & \bzero & \bzero\\
        \diag(\bmu^\star)\bS & 0 & \diag(\bS\bchi^\star+\bT\bd-\bu)
    \end{bmatrix},$$
    and $\bGamma^{-1}$ exists.
\end{proposition}
    The proof of Proposition~\ref{prop:1} can be established (and the linear independence requirement in Assumption~\ref{as:1} relaxed) as a special case of \cite[Theorem 1]{L2O2021}.

Proposition~\ref{prop:1} can be used to evaluate the desired sensitivities $\nabla_{W_{ij}} \bchi^\star$ and $\nabla_{\bu} \bchi^\star$ at the computational cost of solving the linear system~\eqref{eq:prop1}. Specifically, $\nabla_{W_{ij}} \bchi^\star$ can be obtained by solving for $\bdelta_{\bchi}$ in~\eqref{eq:prop1} while setting $\bdelta_{\bW}=\be_i\be_j^\top$ and $\bdelta_{\bu}=\bzero$ on the right hand side. Computing $\nabla_{\bu} \bchi^\star$ is more direct: it is given by the top rows (corresponding to the length of $\bchi$) of $\bGamma^{-1}[\bzero~~-\bI~~ \diag(\bmu^\star)]^\top$. It is apparent that the aforementioned approach of computing $\nabla_{W_{ij}} \bchi^\star$ individually for each $i, j$ is computationally expensive. A useful observation stems from noting that we do not necessarily need $\nabla_{W_{ij}} \bchi^\star$ directly. Rather, we are after computing $\nabla_{\bW}\ell$ when we have already computed $\nabla_{\bchi^\star}\ell$ by combining \eqref{eq:gradL}, \eqref{eq:gradLpgdc}, and \eqref{eq:total}. Capitalizing on this structure, the next result (obtained on the lines of \cite[Theorem 1]{zeng24aL4DC}) helps reduce the computational costs significantly.

\begin{proposition}\label{prop:2}
    Given $\nabla_{\bchi^\star}\ell$, let $\bphi^\top:=-\begin{bmatrix}
        \nabla_{\bchi^\star}\ell&\bzero^\top&\bzero^\top
    \end{bmatrix}\Gamma^{-1}$. Partition the vector as $\bphi^\top=[\bphi_{\bchi}^\top~~\bphi_{\blambda}^\top~~\bphi_{\bmu}^\top]$, such that $(\bphi_{\bchi},~\bphi_{\blambda},~\bphi_{\bmu})$ have dimensions of $(\bchi,\blambda,\bmu)$. Then,
    \begin{equation}\label{eq:prop2}
        \nabla_{\bW}\ell=\blambda^\star\bphi_{\bchi}^\top+\bphi_{\blambda}(\bchi^\star)^\top
    \end{equation}
\end{proposition}
\begin{proof}
    Given $\nabla_{\bchi^\star}\ell$, we have $\nabla_{W_{ij}}\ell=\nabla_{\bchi^\star}\ell\nabla_{W_{ij}}\bchi^\star$. To compute $\nabla_{W_{ij}}\bchi^\star$ we set $\bdelta_{\bu}=\bzero$ and $\bdelta_{\bW}=\be_i\be_j^\top$ in~\eqref{eq:prop1} to obtain
    $$\nabla_{W_{ij}}\bchi^\star=-\begin{bmatrix}
        \bI&\bzero&\bzero
    \end{bmatrix}\Gamma^{-1}\begin{bmatrix}
        (\be_i\be_j^\top)^\top\blambda^\star\\
        \be_i\be_j^\top\bchi^\star\\
        \bzero
    \end{bmatrix}.$$
    Therefore,
    \begin{align}
        \nabla_{W_{ij}}\ell&=-\nabla_{\bchi^\star}\ell\begin{bmatrix}
        \bI&\bzero&\bzero
    \end{bmatrix}\Gamma^{-1}\begin{bmatrix}
        (\be_i\be_j^\top)^\top\blambda^\star\\
        \be_i\be_j^\top\bchi^\star\\
        \bzero
    \end{bmatrix}\notag\\
    &=-\begin{bmatrix}
        \nabla_{\bchi^\star}\ell&\bzero^\top&\bzero^\top
    \end{bmatrix}\Gamma^{-1}\begin{bmatrix}
        \be_j\be_i^\top\blambda^\star\\
        \be_i\be_j^\top\bchi^\star\\
        \bzero
    \end{bmatrix}\notag\\
    &=\bphi^\top\begin{bmatrix}
        \be_j\lambda_i^\star\\
        \be_i\chi_j^\star\label{eq:prop2a}\\
        \bzero
    \end{bmatrix},
    \end{align}
    Using the partitions of $\bphi$, one can rewrite~\eqref{eq:prop2a} as
    \begin{align*}
    \nabla_{W_{ij}}\ell&=\bphi_{\bchi}^\top\be_j\lambda_i^\star+\bphi_{\blambda}^\top\be_i\chi_j^\star\\
    &=\be_i^\top(\blambda^\star\bphi_{\bchi}^\top+\bphi_{\blambda}(\bchi^\star)^\top)\be_j
    \end{align*}
    Putting the gradients for all $i, j$, together yields~\eqref{eq:prop2}.
\end{proof}

With the overall sensitivity analysis structure in place, we use the mini-batch stochastic gradient descent approach summarized in Algorithm~\ref{algo} to obtain optimal PF linearizations.

\begin{algorithm}
\caption{Mini-batch stochastic gradient descent for optimal PF linearization}
\label{algo}
\begin{algorithmic}[1]

    \State \textbf{Given:} scenario set $\{(\mathbf{p}^\mrd_s,\mathbf{q}^\mrd_s)\}_{s=1}^S$, batch size $B$, and max iterations $T$
    \State \textbf{Initialize:} $\bM\leftarrow(\diag(\bx))^{-1}\bA$, $\bgamma\leftarrow\bzero,\bb\leftarrow\bzero$, $t \leftarrow 1$, and $\alpha$.

    \While{$t \le T$}
        \State Set $\alpha \leftarrow \alpha(T-t)/T$
        \State \textbf{Initialize:} 
        $\ell\leftarrow 0, \Delta_{\bM} \leftarrow \bzero,\; \Delta_{\bgamma} \leftarrow \bzero,\; \Delta_{\bb}  \leftarrow \bzero$
        \State Sample mini-batch $\mcB \subseteq \{1,\ldots,S\}$ of size $B$

        \For{each $s \in \mcB$}
            \State Solve (P1) to get $\bp^\mrg_{\mrDC,s}.$  
        
            \State Solve distributed slack based AC PF~\eqref{eq:ACPF}-\eqref{eq:pl} to get $(\bar{\bp}^\mrf_s,\bar{\bp}^\mrg_s)$. 
            \State Compute and accumulate the loss using~\eqref{eq:loss}
            $$\ell\leftarrow \ell+\ell_s(\overline{\mathbf{p}}^g_{s},\overline{\mathbf{p}}^\mrf_{s})$$
            \State Compute and accumulate gradients using~\eqref{eq:gradL}-\eqref{eq:total}, and \eqref{eq:prop1}-\eqref{eq:prop2}
            $$\Delta_\bM\leftarrow \Delta_\bM+\nabla_\bM\ell_s$$
            $$\Delta_{\bgamma}\leftarrow \Delta_{\bgamma}+\nabla_{\bgamma}\ell_s,~\Delta_{\br}\leftarrow \Delta_{\br}+\nabla_{\br}\ell_s$$
        \EndFor

        \State Update Linearization Coefficient:
        $$\bM \leftarrow \bM - \frac{\alpha}{B}\Delta_{\bM},~~
          \bgamma \leftarrow \bgamma -  \frac{\alpha}{B}\,\Delta_{\bgamma},~~
          \bb \leftarrow \bb -  \frac{\alpha}{B}\Delta_{\bb}$$
        \State $t \leftarrow t + 1$
    \EndWhile

\end{algorithmic}
\end{algorithm}


\section{NUMERICAL TESTS}\label{sec:tests}
The performance of the developed PF linearization approach was evaluated using the IEEE 39-bus system. Network parameters, generation limits, and nominal demands were sourced from the MATPOWER ${\verb|casefile|}$~\cite{MATPOWER}. When needed for benchmarking, MATPOWER was used to solve AC OPF instances. Traditional DC OPF instances that involve solving~(P1) with constraints~\eqref{eq:DCpf+} replaced by the classical DC PF model~\eqref{eq:DCpf} were also solved using MATPOWER. For given linearization coefficients $(\bM, \bgamma, \bb)$, the quadratic program~(P1) was solved using the MATLAB-based optimization toolbox CVX and Gurobi. Demand scenarios for obtaining optimal PF linearization and benchmarking performance were generated by scaling the nominal demand at each node of the IEEE 39-bus. The scaling factors were drawn independently from a uniform distribution $\mcU[0.9,1.1]$. A dataset $\{\bp^\mrd_s,\bq^\mrd_s\}_{s=1}^S$, with $S=64$ was used to solve~\eqref{eq:bilevel} via Algorithm~\ref{algo}, with a batch size of 8. A set of 1000 random instances drawn as described above was used for performance evaluation for the tests described next.
   \begin{figure}[t]
      \centering
      \includegraphics[scale=1]{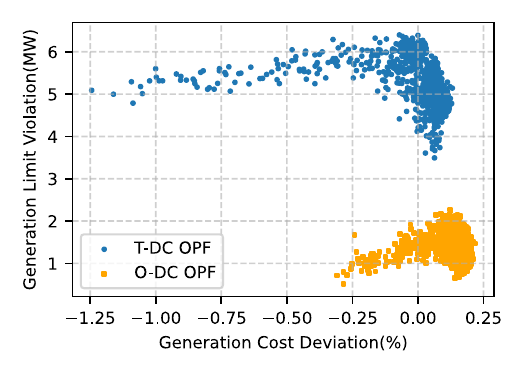}
      \vspace{-3em}
      \caption{Average violation (per test instance) in power generation limits vs increased generation cost as calculated in~\eqref{eq:IC} for the steady-state operating points resulting from traditional and optimized DC OPF.}
      \label{fig:result1}
   \end{figure}

The first set of tests was conducted to assess the cost and feasibility of steady-state operation resulting from the use of traditional DC OPF (T-DC OPF) vs the proposed optimized DC OPF (O-DC OPF). To benchmark cost-optimality, AC OPF was solved for the test scenarios. The percentage increase in the cost of operation for the $k$-th test instance was then evaluated as
\begin{equation}\label{eq:IC}
    \frac{ (\bar{\bp}_k^\mrg-\bp^{\mrg}_{\mrAC,k})^\top\bC(\bar{\bp}_k^\mrg-\bp^{\mrg}_{\mrAC,k})}{(\bp^{\mrg}_{\mrAC,k})^\top\bC\bp^{\mrg}_{\mrAC,k}}\times 100
\end{equation}
Further, the violation in generator active power constraint~\eqref{seq:pg} and line flow limits in~\eqref{seq:pl} were computed by averaging the violations in MW over the count of constraint violations. Figure~\ref{fig:result1} shows the distribution of operating points on the cost-infeasibility plane resulting from T-DC OPF and O-DC OPF when the weight in~\eqref{eq:loss} was set to $w=10$. The obtained distribution shows that the proposed approach significantly decreases constraint violations at a marginal increase in operating cost. Interestingly, both the traditional and optimized DC OPF often result in lower operating costs than the AC OPF. However, such cost reduction stems from unacceptable generation and line limits violations. In practice, these operating scenarios would require additional intervention from grid operators to restore operational feasibility. The importance of reduced violations from O-DC OPF is further accentuated in such cases.

\begin{figure}[t]
      \centering
      \includegraphics[scale=1]{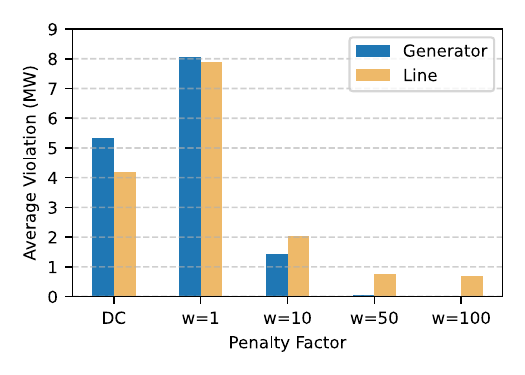}
      \vspace{-3em}
      \caption{Average violation (over all test instances) in generation and line limits for the steady-state operating points resulting from traditional DC OPF and the optimized DC OPF for varying weight parameter $w$.}
      \label{fig:result2}
   \end{figure}

In pursuit of reduced violations, a second set of tests was conducted to study the role of the weight parameter $w$ in traversing the cost-infeasibility trade-off. Figure~\ref{fig:result2} shows that parameter $w$ can indeed be used as a tuning parameter to steer violations to acceptable levels. However, there was some cost for avoiding the violations. The average percentage cost increase~\eqref{eq:IC} over all test instances when using T-DC OPF was found to be $-0.03\%$; the percentage increase for O-DC OPF was $\{-0.21, 0.11, 0.24, 0.37\}$ for $w=\{1,10,50,100\}$, respectively. Figure~\ref{fig:result2} shows that minimal to no violations in generation limits occurred for $w$ greater than 50. Further, it was found that by using the optimal PF linearization coefficients obtained for aggressive weight $w=1000$, the line limit violations were also completely eliminated with an average cost increase of $0.55 \%$. Thus, the proposed approach can be used to avoid the need for operator interventions to restore operational feasibility. 
\section{CONCLUSIONS}
This work has developed a novel application-informed approach for approximating power flow constraints in OPF formulations. The proposed idea is markedly distinct from the conventional practice of assessing PF linearization quality based on inaccuracy with respect to AC PF equations, while being agnostic to the end use. Taking an end-to-end approach, the impact of linearized PF constraints on the DC OPF decisions and subsequent effect on the resulting steady-state grid operating point has been formally characterized. The consequent task of optimizing PF linearization coefficients constitutes a bilevel optimization problem, which is solved using a mini-batch stochastic gradient descent algorithm. A backpropagation-suited sensitivity analysis is carried out for various subsystems to obtain the required gradients. Numerical tests on the IEEE 39-bus system have demonstrated the flexibility harnessed from the proposed approach in traversing the trade-off between cost-optimality and operational feasibility. Specifically, constraint violations, which are a major concern when approximating AC OPF by DC OPF, can be largely avoided at a marginal increase in generation cost. It is worth emphasizing that the developed approach does not increase the complexity of solving DC OPF. Instead, the numerical values of problem parameters are altered while retaining the convex quadratic program structure of traditional DC OPF formulations. Acknowledging the high computational complexity of solving bilevel optimization problems, our future work focuses on algorithmic advancements targeted at the scalability of the developed approach for large-scale systems. One promising direction to achieve scalability is to use the multi-parametric programming approach for sensitivity analysis of quadratic programs as in~\cite{SGKCB2020}. We will also explore ways to adapt our method to larger uncertainties in demand input. Further, we will extend our approach to cater to additional impactful applications in power system operation, planning, and markets.



\bibliography{myabrv,linearization}
\bibliographystyle{IEEEtran}

\end{document}

%% file: StylesMKS.tex
\DeclareMathOperator{\diag}{dg}

\DeclareMathOperator{\dist}{dist}

\newtheorem{assumption}{Assumption}
\newtheorem{proposition}{Proposition}

\newcommand \bzero{\mathbf{0}}
\newcommand \bone{\mathbf{1}}

\newcommand \bb{\mathbf{b}}

\newcommand \bd{\mathbf{d}}
\newcommand \be{\mathbf{e}}

\newcommand \bp{\mathbf{p}}
\newcommand \bq{\mathbf{q}}
\newcommand \br{\mathbf{r}}

\newcommand \bu{\mathbf{u}}
\newcommand \bv{\mathbf{v}}

\newcommand \bx{\mathbf{x}}
\newcommand \by{\mathbf{y}}
\newcommand \bz{\mathbf{z}}
\newcommand \bA{\mathbf{A}}
\newcommand \bB{\mathbf{B}}
\newcommand \bC{\mathbf{C}}

\newcommand \bF{\mathbf{F}}
\newcommand \bG{\mathbf{G}}

\newcommand \bI{\mathbf{I}}
\newcommand \bJ{\mathbf{J}}

\newcommand \bM{\mathbf{M}}

\newcommand \bP{\mathbf{P}}

\newcommand \bR{\mathbf{R}}
\newcommand \bS{\mathbf{S}}
\newcommand \bT{\mathbf{T}}

\newcommand \bW{\mathbf{W}}

\newcommand \bY{\mathbf{Y}}

\newcommand \balpha{\boldsymbol{\alpha}}

\newcommand \bgamma{\boldsymbol{\gamma}}
\newcommand \bdelta{\boldsymbol{\delta}}

\newcommand \btheta{\boldsymbol{\theta}}

\newcommand \blambda{\boldsymbol{\lambda}}
\newcommand \bmu{\boldsymbol{\mu}}

\newcommand \bpi{\boldsymbol{\pi}}

\newcommand \bphi{\boldsymbol{\phi}}
\newcommand \bchi{\boldsymbol{\chi}}

\newcommand \bGamma{\mathbf{\Gamma}}

\newcommand \bPi{\mathbf{\Pi}}

\newcommand \bPsi{\mathbf{\Psi}}

\newcommand \mrd{\mathrm{d}}

\newcommand \mrg{\mathrm{g}}

\newcommand \mrf{\mathrm{f}}

\newcommand \mrAC{\mathrm{AC}}
\newcommand \mrDC{\mathrm{DC}}

\newcommand \mcA{\mathcal{A}}
\newcommand \mcB{\mathcal{B}}

\newcommand \mcE{\mathcal{E}}

\newcommand \mcN{\mathcal{N}}

\newcommand \mcU{\mathcal{U}}

\newcommand \mcZ{\mathcal{Z}}

